\documentclass[12pt]{article}
\usepackage{amssymb}
\usepackage{amsthm}
\usepackage{amsmath}
\usepackage{graphicx}
 \newtheorem{thm}{Theorem}[section]
 
 \newtheorem{lem}[thm]{Lemma}
 
 \theoremstyle{definition}

 \numberwithin{equation}{section}

\begin{document}
\title{Global Solutions for Micro-Macro Models of Polymeric Fluids}
\author{Zhen Lei\footnote{Shanghai Key Laboratory for
Contemporary Applied Mathematics; Key Laboratory of Nonlinear
Mathematical Models and Methods of Ministry of Education; School
of Mathematical Sciences, Fudan University, Shanghai 200433,
China. {\it Email:
 leizhn@yahoo.com, zlei@fudan.edu.cn}}\and
 Yun Wang\thanks{The Institute of Mathematical Sciences,
  The Chinese University of Hong Kong, Shatin, N.T., Hong Kong ({\tt
  ywang@math.cuhk.edu.hk})}}
\date{\today}
\maketitle

\begin{abstract}
We provide a new proof for the global well-posedness of systems
coupling fluids and polymers in two space dimensions. Compared to
the well-known existing method based on the losing a priori
estimates, our method is more direct and much simpler. The
co-rotational FENE dumbbell model and the coupling Smoluchowski
and Navier-Stokes equations are studied as examples to illustrate
our main ideas.
\end{abstract}

\maketitle



\section{Introduction}
The dynamics of many polymeric fluids are described by two-scale
micro-macro models. The systems usually consist of a macroscopic
momentum equation and a microscopic Fokker-Planck type equation.
The fluid is described by the macroscopic equation (sometimes the
incompressible Navier-Stokes equations), with an induced elastic
stress. The stress is the micro-macro interaction. The particles
in the system are represented by a probability distribution
$\psi(t,x,R)$ or $\psi(t,x,m)$ that depends on time $t$,
macroscopic physical location $x$ and particle configuration $R$
or $m$. The Lagrangian transport of the particles is modeled using
a Taylor expansion of the velocity field, which accounts for a
drift term that depends on the spatial gradient of velocity. The
system attempts to describe the behavior of this complex mixture
of polymers and fluids. For more physical and mechanical
backgrounds, see \cite{BCA, DE}.

The FENE (Finite Extensible Nonlinear Elastic) dumbbell model is
one of the typical and extensively studied micro-macro models. In
this model a polymer is idealized as an "elastic dumbbell"
consisting of two "beads" joined by a spring which can be modeled
by a vector $R$. Mathematically, this system reads
\begin{equation}\label{1.1} \left\{
\begin{array}{l}
\displaystyle\partial_t u + (u\cdot\nabla)u + \nabla
p = \nu \Delta u + {\rm div}~\tau,\quad x \in \mathbb{R}^2,\\[3mm]
{\rm div}~u =0,\quad x \in \mathbb{R}^2,\\[3mm]
\partial_t \psi + (u\cdot \nabla)\psi = {\rm div}_R \big[-W(u) \cdot R\psi + \beta
\nabla_R\psi\\
\quad\ \  \ \ \ \ \ \ \ \ \ + \nabla_R\mathcal{U} \psi\big],\quad
(x, R) \in
\mathbb{R}^2 \times B(0,R_0),\\[3mm]
(\nabla_R \mathcal{U}\psi + \beta \nabla_{R}
\psi)\cdot \textbf{n} = 0,\quad \mbox{on}\ \partial B(0,R_0),\\[3mm]
t=0: u(t, x) = u_0(x),\ \ \ \psi(t, x, R) = \psi_0(x,R).
\end{array}\right.
\end{equation}
In the above system, $u=u(t,x)$ denotes the velocity field of the
fluid, $p=p(t,x)$ denotes the pressure, $\psi(t,x,R)$ is the
distribution function for the internal configuration, $\nu>0$ is
the viscosity of the fluid and $\beta$ is related to the
temperature of the system.  Moreover, the spring potential
$\mathcal{U}$ and the induced elastic stress $\tau$ is given by
\begin{equation}\label{1.1-1}
\mathcal{U}(R) = -k\ln(1-|R|^2/|R_0|^2),\quad \tau_{ij} =
\int_{B(0,R_0)} (R_i \otimes \nabla_{R_j} \mathcal{U})\psi(t,x, R)
dR.
\end{equation}
Here $k > 0$ is a constant. The boundary condition  insures the
conservation of the polymer density. Assume that $W(u)=
\frac{\nabla u - (\nabla u)^t}{2}$, which corresponds to the
co-rotational case. For simplicity of writing, assume that
$\beta=1$ and $R_0 =1$ and denote $B(0,R_0)$ by $B$. In what
follows, without special claim, $\nabla$ represents $\nabla_x$,
and ${\rm div}$ represents ${\rm div}_x$.

The Smoluchowski equation coupled with the incompressible
Navier-Stokes equations is another extensively studied micro-macro
model. The Smoluchowski equation describes the temporal evolution
of the probability distribution function $\psi$ for directions of
rod-like particles in a suspension. Mathematically, the system
reads
\begin{equation}\label{1.2} \left\{
\begin{array}{l}
\displaystyle\partial_t u + (u\cdot\nabla)u + \nabla p
= \nu \Delta u + {\rm div}~\tau, \quad x \in \mathbb{R}^2,\\[3mm]
{\rm div}~u =0, \quad x \in \mathbb{R}^2,\\[3mm]
\partial_t \psi + (u\cdot \nabla)\psi + {\rm div}_g(G(u, \psi)\psi)- \Delta_g \psi=0,\quad (x, m) \in \mathbb{R}^2\times M,\\[2mm]
t=0: u(t, x) = u_0(x),\ \ \ \psi(t, x, m) = \psi_0(x,m).
\end{array}\right.
\end{equation}
Here $M$ is a $d$-dimensional smooth compact Riemannian manifold
without boundary of
 and $dm$ is the Riemannian volume element of $M$,
$u=u(t,x)$ and $p=p(t,x)$ denote the velocity field and the
pressure of fluid respectively, $\psi=\psi(t,x,m)$ is the
distribution function, $G(u, \psi)= \nabla_g \mathcal{U} + W$
stands for a meanfield potential resulting from the excluded
volume effects due to steric forces between molecules with $W=
c_\alpha^{ij} (m)\partial_j u_i$. Besides, the added stress tensor
$\tau$ and the potential $\mathcal{U}$ are given by
\begin{equation}\label{1.2-1} \left\{
\begin{array}{l}
\mathcal{U}(t,x,m) = \displaystyle \int_M K(m,q) \psi(t,x,
q)dq,\\[3mm]
\tau_{ij}(t,x) =\displaystyle \int_M \gamma_{ij}^{(1)}(m)
\psi(t,x,m)dm\\[3mm] \quad\ \ \ \ \displaystyle+ \int_M\int_M
\gamma_{ij}^{(2)}(m_1, m_2)\psi(t,x,m_1)\psi(t,x,m_2) dm_1dm_2.
\end{array}\right.
\end{equation}
Here the kernel $K$ is a smooth symmetric function defined on
$M\times M$. $\gamma_{ij}^{(1)}$ and $\gamma_{ij}^{(2)}$ are
smooth, time independent and $x$ independent.

At present there have been extensive and systematical studies on
the existence and regularity theories of those 2D micro-macro
models of polymeric fluids \cite{BC1994, C, CM, CFTZ, CMa, LLZ,
Ma}. For example, the first global well-posedness result for FENE
\eqref{1.1} was derived by Lin, Zhang and Zhang \cite{LZZ} for
$k>6$.  Masmoudi \cite{Ma} extends it to the case of $k>0$ by a
crucial observation on the linear operator. Very recently the
global existence of weak solutions to the FENE dumbbell model of
polymeric flows for a very general class of potentials was also
obtained by Masmoudi \cite{Masmoudi2}. The global well-posedness
of nonlinear Fokker-Planck system coupled with Navier-Stokes
equations \eqref{1.2} in 2D has been proven by Constantin and
Masmoudi in \cite{CMa}. When the nonlinear Fokker-Planck equation
is driven by a time averaged Navier-Stokes system in 2D, global
well-posedness has been obtained by
Constantin-Fefferman-Titi-Zarnescu \cite{CFTZ}.

Most proofs of the above global well-posedness theorems are based
on an important analytic technique called losing a priori estimate
in the spirit of Bahouri-Chemin \cite{BC1994} and Chemin-Masmoudi
\cite{CM}. In \cite{LMZ}, we studied the blow-up criteria of a
macroscopic viscoelastic Oldroyd-B system avoiding using the
losing a priori estimates. The main purpose of this paper is to
extend the method in \cite{LMZ} to micro-macro models and provide
a new proof for  global well-posedness of the co-rotational FENE
dumbbell model \eqref{1.1} and the coupling Smoluchowski and
incompressible Navier-Stokes equations (\ref{1.2}). Compared to
the proofs of the theorems in \cite{Ma} and \cite{CMa}, which are
based on the technique of losing a priori estimates, ours are
direct and much simpler.

For the co-rotational FENE model (\ref{1.1}), we have

\begin{thm}\label{thm 1.1}
Assume that $u_0 \in H^s(\mathbb{R}^2), (s>2)$, $\psi_0 \in
H^s(\mathbb{R}^2; \mathcal{L}^r \cap \mathcal{L}^2)$ for some
$r\geq 2$ such that $(r-1)k>1$ with $\int_{B} \psi_0 dR =1, a.e.$
in $x$. Then there exists a unique global solution $(u, \psi)$ of
the FENE model (\ref{1.1}) in $C([0, \infty); H^s) \times C([0,
\infty); H^s(\mathbb{R}^2; \mathcal{L}^2)).$ Moreover, $u\in
L_{loc}^2 (0, \infty; H^{s+1})$ and $\psi \in L^2_{loc}(0, \infty;
H^s(\mathbb{R}^2; \mathcal{L}^{1,r})).$
\end{thm}
For the definition of $\mathcal{L}^r$ and $\mathcal{L}^{1,r}$,
please refer to section 2.

Similarly, for the coupling Smoluchowski and Navier-Stokes system
(\ref{1.2}), we have

\begin{thm}\label{thm 1.2}
Take $u_0\in W^{1+\epsilon,r}(\mathbb{R}^2)\cap L^2(\mathbb{R}^2)$
and $\psi_0\in W^{1,r}(\mathbb{R}^2, H^{-s}(M)),$ for some $r> 2$,
$\epsilon >0$, $s>\frac{d}{2} + 1$ and $\psi_0\geq 0$. $\int_M
\psi_0 dm\in L^1(\mathbb{R}^2) \cap L^\infty(\mathbb{R}^2)$. Then
(\ref{1.2}) has a global solution in $u\in
L_{loc}^{\infty}(0,\infty; W^{1,r})\cap L_{loc}^2(0, \infty;
W^{2,r})$ and $\psi \in L_{loc}^\infty(0,\infty;
W^{1,r}(\mathbb{R}^2; H^{-s}(M)))$. Moreover, for $T>T_0>0$, we
have $u\in L^\infty((T_0, T); W^{2-\epsilon, r}).$
\end{thm}

We end this introduction by mentioning some other results on
micro-macro models. Global existence of weak solutions can be
found in \cite{BSS, LM} and local existence of strong solutions
are studied in \cite{ELZ, JLL, ZZ}. For macroscopic models, we
refer the readers to \cite{Lei1, Lei2, LZ, LLZ1, LLZh, CZ} as
references.

The paper is organized as follows. In section 2, we give some
preliminaries. Then we give some a priori estimates for the FENE
model \eqref{1.1} in section 3 and the coupling Smoluchowski and
Navier-Stokes equations \eqref{1.2} in section 4. The a priori
estimates obtained in sections 3 and 4 are enough to get the
global existence of systems \eqref{1.1} and \eqref{1.2} and prove
Theorem 1.1 and Theorem 1.2 \cite{Ma, CMa}.

\section{Definitions and Useful Lemmas}

We will use the Littlewood-Paley decomposition in the following
sections. Define $\mathcal{C}$ to be the ring
$$\mathcal{C} = \{\xi\in \mathbb{R}^2: \frac{3}{4} \leq |\xi| \leq \frac{8}{3}\},$$
and define $\mathcal{D}$ to be the ball
$$\mathcal{D} = \{\xi\in \mathbb{R}^2: |\xi|\leq \frac43\}.$$
Let $\chi$ and $\varphi$ be two smooth nonnegative radial
functions supported respectively in $\mathcal{D}$ and
$\mathcal{C}$, such that
$$\chi(\xi) + \sum_{q\geq 0} \varphi(2^{-q} \xi) =1\ \ \mbox{for}\ \xi\in \mathbb{R}^2
,\ \ \mbox{and}\ \sum_{q\in \mathbb{Z}} \varphi(2^{-q}\xi)=1\ \
\mbox{for}\ \xi\in \mathbb{R}^2\setminus \{0\}.$$

Let us denote by $\mathcal{F}$ the Fourier transform on
$\mathbb{R}^2$ and denote $$h= \mathcal{F}^{-1}\varphi,\ \ \ \ \ \
\tilde{h} = \mathcal{F}^{-1}\chi.$$ The frequency localization
operator is defined by
$$\Delta_q u =  \mathcal{F}^{-1}\left[\varphi(2^{-q}\xi)\mathcal{F}(u)\right] =
2^{2q} \int_{\mathbb{R}^2} h(2^q y) u(x-y) dy, $$ and
$$S_q u = \mathcal{F}^{-1}\left[\chi(2^{-q}\xi)\mathcal{F}(u)\right] =
2^{2q} \int_{\mathbb{R}^2} \tilde{h}(2^q y) u(x-y) dy. $$

Hence, for $s<2/p$ or $s=2/p$ and $r=1$, the homogeneous Besov
space $\dot{B}_{p, r}^s$ is defined as the closure of compactly
supported smooth functions under the norm
$\|\cdot\|_{\dot{B}_{p,r}^s},$
$$\|u\|_{\dot{B}_{p,r}^s} = \left\|
(2^{qs} \|\Delta_q u\|_{L^p})_{q\in
\mathbb{Z}}\right\|_{l^r(\mathbb{Z})}.$$ when $p=r=\infty$,
$s=k+\alpha$ where $k\in \mathbb{N}$ and $\alpha\in (0,1)$, then
$\dot{B}_{p,r}^s$ turns to be the homogeneous H\"{o}lder space
$\dot{C}^{k+\alpha}$.

Another kind of space to be used is $\tilde{L}^p(t_1, t_2;
\dot{C}^r)$, which is the space of distributions $u$ such that
$$\|u\|_{\tilde{L}^p(t_1, t_2; \dot{C}^r)} \triangleq \sup_{q\in \mathbb{Z}} 2^{qr} \|\Delta_q u\|_{L^p(t_1, t_2; L^\infty)}
< \infty.$$

For the FENE model, let
$$\psi_\infty (R) =
\frac{e^{-\mathcal{U}(R)}}{\int_B e^{-\mathcal{U}(R)} dR} = c (1-
|R|^2)^k,$$ where the constant $c$ is given such that $\int_B
\psi_\infty dR =1.$ In fact, $(u,\psi)=(0, \psi_\infty)$ defines a
stationary solution of (\ref{1.1}).

For $r\geq 1$, denote $\mathcal{L}^r$ and $\mathcal{L}^{1,r}$ the
weighted spaces
$$\mathcal{L}^r =\left\{\psi: \|\psi\|_{\mathcal{L}^r}^r =
\int_B \psi_\infty \left|\frac{\psi}{\psi_\infty}\right|^rdR <
\infty\right\}$$$$\mathcal{L}^{1, r} =\left\{\psi\in
\mathcal{L}^r: |\psi|_{\dot{\mathcal{L}}^{1,r}}^r = \int_B
\psi_\infty
\left|\nabla_R\left(\frac{\psi}{\psi_\infty}\right)\right|^rdR <
\infty\right\}.$$

We will need to use the following well-known inequalities.

\begin{lem}[{\rm Bernstein Inequalities \cite{Ch}}]
\label{lemma 2.1} For $s\in \mathbb{R}$, $1\leq p\leq r\leq
\infty$ and $q\in \mathbb{Z}$, one has
$$\|\Delta_q u\|_{L^r(\mathbb{R}^d)}\leq C\cdot 2^{d(\frac1p-\frac1r)q}\|\Delta_q u\|_{L^p(\mathbb{R}^d)},$$
$$c 2^{qs} \|\Delta_q u\|_{L^p} \leq \|\nabla^s \Delta_q u\|_{L^p}
\leq C  2^{qs} \|\Delta_q u\|_{L^p},$$
$$\||\nabla|^s S_q u\|_{L^p} \leq C 2^{qs}\|u\|_{L^p},$$
$$c e^{-C 2^{2q} t } \|\Delta_q u\|_{L^\infty} \leq \|e^{t\Delta} \Delta_q u\|_{L^\infty}
\leq C e^{-c 2^{2q} t}\|\Delta_q u\|_{L^\infty}.$$ Here $C$ and
$c$ are positive constants independent of $s, p$ and $q$.
\end{lem}

We will also need the following lemma, whose proof  can be found
in \cite{LMZ}.

\begin{lem} \label{lemma 2.2}
Assume that $\beta >0$. Then there exists a positive constant
$C>0$ such that
$$\begin{array}{l}\displaystyle \int_t^T \|\nabla g(s, \cdot)\|_{L^\infty}ds \leq
C\left(
1+ \int_t^T \|g(s,\cdot)\|_{L^2}ds\right. \\[3mm]\ \ \ \ \ \ \ \ \ \ \ \displaystyle+\left.\sup_q \int_t^T \|\Delta_q
\nabla g(s,\cdot)\|_{L^\infty}ds \ln (e+ \int_t^T \|\nabla
g(s,\cdot)\|_{\dot{C}^\beta}ds)\right).\end{array}$$
\end{lem}

\section{Proof of Theorem 1.1}
Since local existence of smooth solutions has been derived by N.
Masmoudi \cite{Ma}, here we only focus on the a priori estimate
which is sufficient for proving Theorem \ref{thm 1.1}. As
explained in \cite{LZZ} or \cite{Ma}, to get the global existence
we just need to control the $L^\infty$-norm of $\nabla u$, i.e.,
$\|\nabla u\|_{L^\infty}$. Define the flow associated with $u$ by
$\Phi(t,x)$, which means $\Phi$ satisfies the ODEs,
\begin{equation}\left\{
\begin{array}{l} \partial_t \Phi (t, x) = u (t, \Phi(t,x)),\\
\Phi(t=0, x) = x.\end{array} \right.
\end{equation}\\
\textbf{Step I}~~\textbf{Uniform estimates for $\psi$ and $\tau$
with respect to $t$.}

Due to the special structure of the equation about $\psi(t, x,
R)$, we have the following bounded estimates of $\psi$. For $r>1$,
multiplying the third equation of (\ref{1.1}) by
$r\left|\frac{\psi}{\psi_\infty}\right|^{r-2}\frac{\psi}{\psi_\infty},$
and integrating over $B$, then we get
$$\displaystyle \partial_t \int_{B} \psi_\infty \left|\frac{\psi}{\psi_\infty}\right|^rdR
 + u\cdot \nabla \int_{B} \psi_\infty \left|\frac{\psi}{\psi_\infty}\right|^rdR
 = - \frac{4(r-1)}{r} \int_{B} \psi_\infty \left|\nabla_R \left(
 \frac{\psi}{\psi_\infty}\right)^{\frac{r}{2}}\right|^2 dR .$$
Therefore,
$$\displaystyle \int_{B}\psi_\infty \left|\frac{\psi}{\psi_\infty}\right|^r (t,\Phi (t, x), R) dR
\leq \int_{B} \psi_\infty
\left|\frac{\psi_0}{\psi_\infty}\right|^r(x, R) dR.$$ Since the
flow is incompressible, then
\begin{equation}\label{3-1}\|\psi\|_{L^\infty_{t,x} (\mathcal{L}^r)}\leq
\|\psi_0\|_{L^\infty_x (\mathcal{L}^r)},\ \ \
 \|\psi\|_{L^\infty_t (L^2_x (\mathcal{L}^r))} \leq
 \|\psi_0\|_{L^2_x(\mathcal{L}^r)}.\end{equation}

 To estimate $\tau$, we need a lemma.

 \vspace{2mm}\begin{lem}
 \label{lemma 3.1}For any $p$ such that $pk>1$, it
holds that
$$\int_{B} \frac{|\psi|}{1-|R|} dR \leq C\left(
\int_{B}
\frac{|\psi|^{p+1}}{\psi_\infty^p}\right)^{\frac{1}{p+1}}=
\|\psi\|_{\mathcal{L}^{p+1}}.$$
\end{lem}

\begin{proof} By the H\"{o}lder's inequality,
$$\begin{array}{ll}\displaystyle\int_{B} \frac{|\psi|}{1-|R|} dR & \leq \displaystyle C \int_{B}
\frac{1}{(1-|R|)^{1-kp/(p+1)}}
\frac{|\psi|}{\psi_\infty^{p/(p+1)}}dR\\[4mm]
& \leq \displaystyle C\left(\int_{B} \frac{1}{(1-|R|)^{1+ 1/p
-k}}\right)^{\frac{p}{p+1}} \left(\int_{B}
\frac{|\psi|^{p+1}}{\psi_\infty^p}dR\right)^{\frac{1}{p+1}}.
\end{array}$$
Since $pk>1$, the result follows.
\end{proof}

\vspace{2mm}Noting that
$$\tau_{ij}(t,x) = \int_{B} (R_i \otimes \nabla_{R_j} \mathcal{U} )\psi(t,x,R) dR,$$
one has
$$\begin{array}{ll}
|\tau(t,x)|& \leq \displaystyle C\int_{B} |\nabla_{R}
\mathcal{U}|\cdot |\psi(t,x,R) |dR\\
[3mm] & \leq Ck\displaystyle \int_{B}
\frac{|R|}{1-|R|^2}\cdot|\psi(t,x,R)|dR\\
[3mm]& \leq C \displaystyle \int_{B} \frac{|\psi(t,x,R)|}{1-|R|}dR
\\[3mm]& \leq C\|\psi(t,x, \cdot)\|_{\mathcal{L}^r},
\end{array}$$
where in the last step we used Lemma \ref{lemma 3.1}.

Hence, we have
\begin{equation}\label{3-2}
\|\tau\|_{L^\infty(0, T; L^2)}\leq
C\|\psi_0\|_{L^2_x(\mathcal{L}^r)},\ \ \ \ \|\tau\|_{L^\infty(0,T;
L^\infty)}\leq C\|\psi_0\|_{L^\infty_x(\mathcal{L}^r)}.
\end{equation}
\\\textbf{Step II}~~\textbf{A priori estimates for $u$.}

We need a useful lemma whose proof was established by Chemin and
Masmoudi \cite{CM} (see also \cite{LMZ}).

\vspace{2mm}\begin{lem}[{\rm Chemin-Masmoudi}]
 \label{lemma 3.2}Let $v$ be a solution of the Navier-Stokes equations with
initial data $v_0\in L^2(\mathbb{R}^2)$, and an external force
$f\in \tilde{L}^1(0,T; \dot{C}^{-1})\cap L^2(0,T; \dot{H}^{-1})$:
\begin{equation}\label{3-3}
\left\{
\begin{array}{l}
\partial_t v -\Delta v + (v\cdot \nabla)v + \nabla p = f,\ \ \ \mbox{in}\ \mathbb{R}^2\times (0,T)\\
{\rm div}~v=0,\ \ \ \mbox{in}\ \mathbb{R}^2\times (0,T)\\
v(t=0, x) = v_0(x),\ \ \ \mbox{in}\ \mathbb{R}^2
\end{array}
\right.
\end{equation}
Then we have the following a priori estimates,
$$\|v\|_{L^\infty(0,T; L^2)}^2 + 2\|\nabla v\|_{L^2(0,T; L^2)}^2 \leq
\|v_0\|_{L^2}^2 + \|f\|_{L^2(0,T; \dot{H}^{-1}))}^2,$$ and
$$\begin{array}{l}\|v\|_{\tilde{L}^1(0,T; \dot{C}^1)} \leq  \displaystyle C \left( \sup_q \|\Delta_q
v_0\|_{L^2}(1- \exp\{-c 2^{2q}T\})\right. \\ \ \ \ \ \ \  +
\displaystyle \left.(\|v_0\|_{L^2}+ \|f\|_{L^2(0,T; \dot{H}^{-1}})
\|\nabla v\|_{L^2(0, T; L^2)}^2 + \sup_q \int_0^T \|2^{-q}\Delta_q
f(s)\|_{L^\infty} ds\right).\end{array}$$
 Furthermore, if $f\in
L^1(0,T; \dot{C}^{-1})$, then $\forall \epsilon>0,$ there exists
$t_0(\epsilon)\in (0,T)$ such that
$$\|v\|_{\tilde{L}^1(t_0, T; \dot{C}^1)}\leq \epsilon.$$
\end{lem}

\vspace{2mm}Particularly for our problem, since we have shown that
$\tau \in L^\infty(0,T; L^2)\cap L^\infty(0,T; L^\infty)$,
applying Lemma \ref{lemma 3.2}, we know that
$$u\in L^\infty(0,T; L^2)\cap L^2(0,T; \dot{H}^1)\cap \tilde{L}^1(0,T; \dot{C}^1)$$
and \begin{equation}\label{3-add}\forall \epsilon >0, \exists\
t_0(\epsilon)\in (0,T), \ \ \mbox{such that} \ \
\|u\|_{\tilde{L}^1(t_0, T; \dot{C}^1)}\leq \epsilon.\end{equation}

\vspace{2mm}\textbf{Step III~~H\"{o}lder estimates for $u$}

For $0\leq t < T,$ choose some $\alpha$ satisfying $0<\alpha <
\min{\{s-2,  1\}}$, define
$$N_q^r(t,x) = \int_{B} \psi_\infty \left|\frac{\Delta_q \psi(t,x,R)}{\psi_\infty}\right|^r dR
= \|\Delta_q \psi\|_{\mathcal{L}^r}^r(t,x),$$$$A(t)  = \sup_{0\leq
s <t }\|u(s,\cdot)\|_{\dot{C}^{1+\alpha}},\ \ \ B(t) = \sup_{0\leq
s <t}\|\tau(s,\cdot)\|_{\dot{C}^\alpha}$$
$$D(t) = \sup_{0\leq s < t}\sup_{q\in \mathbb{Z}}2^{\alpha q}\|N_q(s, \cdot)\|_{L^\infty}.$$
Here $\Delta_q$ is the frequency operator with respect to $x$.

Before the detailed estimates, we will introduce an inequality for
later use, which can be considered as an extension of H\"{o}lder
inequality.

 \vspace{2mm}\begin{lem}\label{lemma 3.3}For any $u\in
L^4(\mathbb{R}^2) \cap \dot{C}^{1+\alpha}(\mathbb{R}^2)$, there
holds that
$$\|u\otimes u\|_{\dot{C}^{\frac12+\alpha}} \leq C\|u\|_{L^4}\cdot \|u\|_{\dot{C}^{1+\alpha}},$$
with some constant $C$ independent of $u$.
\end{lem}

\begin{proof}
For any $q\in \mathbb{Z}$, using Bony's para-product
decomposition\cite{B}, we have
$$\begin{array}{ll}
\|\Delta_q (u\otimes u)\|_{L^\infty} \cdot 2^{(\frac12 + \alpha)
q} & \leq \displaystyle 2\sum_{|p-q|\leq 5} \|\Delta_q (S_{p-1}u
\otimes \Delta_p u)\|_{L^\infty}\cdot 2^{(\frac12 + \alpha)q}\\[3mm] &\
\ \  + \displaystyle\sum_{p\geq q-3}\sum_{|p-r|\leq 1} \|\Delta_q
(\Delta_p u \otimes \Delta_r u)\|_{L^\infty}\cdot
2^{(\frac12+\alpha)q}\\[3mm]
& \triangleq  2I_1 + I_2.
\end{array}$$
$I_1$ can be estimated as
$$\begin{array}{ll} I_1 & \leq  \displaystyle C\sum_{|p-q|\leq 5}
\|S_{p-1}u \otimes \Delta_p u\|_{L^4}\cdot 2^{(1+\alpha)q}\\[5mm]
 & \leq \displaystyle C\sum_{|p-q|\leq 5}\|S_{p-1} u\|_{L^4}
 \cdot \|\Delta_p u\|_{L^\infty}\cdot 2^{(1+\alpha)q}\\[5mm]
 & \leq \displaystyle C \sum_{|p-q|\leq 5} \|u\|_{L^4}
 \cdot 2^{(1+\alpha)p}\|\Delta_p u\|_{L^\infty} \cdot 2^{(1+\alpha)(q-p)}\\[5mm]
 & \leq \displaystyle C\|u\|_{L^4} \cdot \|u\|_{\dot{C}^{1+\alpha}},\end{array}$$
 here the first inequality is due to Lemma \ref{lemma 2.1}.

While $I_2$ can be estimated as
$$\begin{array}{ll}
I_2 & = \displaystyle \sum_{p\geq q-3}\sum_{|p-r|\leq 1}\|\Delta_q
(\Delta_p u\otimes \Delta_r u)\|_{L^\infty}\cdot 2^{(\frac12
+\alpha)
q}\\[5mm]
& \leq C\displaystyle \sum_{p\geq q-3}\sum_{|p-r|\leq 1}\|\Delta_p
u\|_{L^\infty} \cdot \|\Delta_r u\|_{L^4}\cdot
2^{(\frac12+\alpha)q}\cdot 2^{\frac12 r}\\[5mm]
& \leq C\displaystyle \sum_{p\geq q-3} 2^{(1+\alpha)p}\|\Delta_p
u\|_{L^\infty} \cdot \|u\|_{L^4}\cdot 2^{(\frac12+\alpha)(q-p)}\\[5mm]
& \leq C\|u\|_{\dot{C}^{1+\alpha}} \cdot\|u\|_{L^4},
\end{array}$$
here in the second inequality we also used Lemma \ref{lemma 2.1}.
These above estimates complete the proof of Lemma \ref{lemma 3.3}.
\end{proof}

\vspace{2mm}First, applying $\Delta_q$ to the first equation of
the FENE system, then we obtain
$$\partial_t \Delta_q u- \nu \Delta \Delta_q u + \nabla \Delta_q p = \nabla\cdot \Delta_q
(\tau - u\otimes u),$$ hence
$$\displaystyle \Delta_q u = e^{\nu t\Delta}\Delta_q u_0 +  \int_0^t e^{\nu (t-s)\Delta}
\mathbb{P}(\Delta_q \nabla\cdot (\tau - u\otimes u))ds,$$ where
$\mathbb{P}$ is the Helmoltz-Weyl projection operator.

\begin{equation}\label{3-addadd}\begin{array}{l} \ \ \ 2^{q(1+\alpha)}\|\Delta_q u\|_{L^\infty}(t)\\[2mm]\leq
C\displaystyle e^{-c\nu2^{2q}t}\|\Delta_q u_0\|_{L^\infty}
2^{q(1+\alpha)}+ C\int_0^t e^{-c\nu2^{2q} (t-s)}
2^{q(1+\alpha)}\|\Delta_q \nabla\cdot (\tau - u\otimes
u)\|_{L^\infty}ds\\[3mm] \leq  \displaystyle C\|u_0\|_{\dot{C}^{1+\alpha}} +  C\int_0^t
e^{-c 2^{2q} (t-s)} \cdot 2^{q(1+\alpha)}\left(\|\nabla\cdot
\Delta_q \tau\|_{L^\infty}+ \|\nabla\cdot \Delta_q( u\otimes
u)\|_{L^\infty}\right)ds
\end{array}\end{equation}

Applying Lemma \ref{lemma 2.1}, we obtain $$\begin{array}{ll}&
\displaystyle\int_0^t e^{-c 2^{2q}(t-s)}2^{q(1+\alpha)}
\|\nabla\cdot\Delta_q \tau\|_{L^\infty}ds\\[3mm]
 \leq & C\displaystyle\int_0^t e^{-c 2^{2q}(t-s)} 2^{q(1+\alpha)} 2^q \|\Delta_q
 \tau\|_{L^\infty}ds\\[3mm]
 \leq & C\displaystyle\int_0^t e^{-c 2^{2q} (t-s)}2^{2q} \cdot
 2^{\alpha q} \|\Delta_q \tau\|_{L^\infty}ds \\[3mm]
 \leq & CB(t).
 \end{array}$$
 On the other hand, applying Lemma \ref{lemma 2.1} again, we obtain
 $$\begin{array}{ll} & 2^{q(1+\alpha)} \displaystyle \int_0^t e^{-c 2^{2q}(t-s)} \|\nabla\cdot
 \Delta_q (u\otimes u)\|_{L^\infty} ds \\[3mm]
 \leq & C\displaystyle \int_0^t e^{-c 2^{2q}(t-s)}2^{q(2+\alpha)}\|\Delta_q (u\otimes
 u)\|_{L^\infty}ds\\[3mm]
\leq & C\displaystyle\int_0^t e^{-c 2^{2q}(t-s)}2^{\frac32
q}\|u\otimes u\|_{\dot{C}^{\frac12 +\alpha}}ds\\[3mm]
 \leq & C \displaystyle\int_0^t e^{-c 2^{2q}(t-s)} 2^{\frac32 q} \|u\|_{L^4}\cdot
 \|u\|_{\dot{C}^{1+\alpha}} ds\\[3mm]
 \leq &\displaystyle C\left(\int_0^t \|u\|_{L^4}^4 \|u\|_{\dot{C}^{1+\alpha}}^4
 ds\right)^{\frac14} \cdot \left(\int_0^t e^{-c2^{2q}(t-s)}2^{2q} ds\right)^{\frac34}\\[3mm]
 \leq &C \displaystyle \left(\int_0^t \|u\|_{L^2}^2 \|\nabla u\|_{L^2}^2
  \|u\|_{\dot{C}^{1+\alpha}}^4 ds\right)^{\frac14},
 \end{array}$$
where we used Lemma \ref{lemma 3.3} and interpolation inequality
for the third
 and last inequality respectively.

 Taking the supreme of both sides of (\ref{3-addadd})with respect to $q$, one gets
 $$\displaystyle \|u(t)\|_{\dot{C}^{1+\alpha}}^4 \leq C(\|u_0\|_{\dot{C}^{1+\alpha}} + B(t))^4
 + C\left( \int_0^t \|u(s)\|_{L^2}^2 \|\nabla u(s)\|_{L^2}^2 \|u(s)\|_{\dot{C}^{1+\alpha}}^4 ds\right).$$
By the Gronwall's inequality and the estimates in Step II, we have
\begin{equation}\label{3-4}
A(t) \leq C(\|u_0\|_{\dot{C}^{1+\alpha}} + B(t))\leq C(1+B(t)).
\end{equation}
As a result of the relationship between $\tau$ and $\psi$ and
Lemma \ref{lemma 3.1},
$$\begin{array}{ll} B(t) & =\displaystyle
\sup_{0\leq s <t} \sup_q 2^{\alpha q}\|\Delta_q \tau(s, \cdot)
\|_{L^\infty}\\[2mm]
& \leq \displaystyle \sup_{0\leq s <t}\sup_q 2^{\alpha q} \left\|
\int_B |\Delta_q\psi| \cdot |\nabla_R \mathcal{U}| dR
\right\|_{L^\infty}
\\[3mm] & \displaystyle \leq  C \sup_{0\leq s <t}\sup_q 2^{\alpha q} \|\Delta_q
\psi\|_{L_x^\infty (\mathcal{L}^r)}(s) \\[2mm]
& = C D(t)
\end{array} $$

Combining the above inequality and (\ref{3-4}),
\begin{equation}\label{3-5}
A(t) \leq C(1+ D(t)).
\end{equation}

\vspace{2mm}\textbf{Step IV~~H\"{o}lder Estimates of $\psi$.}

The remaining part is to estimate the $\dot{C}^\alpha$-norm of
$\psi$.
 Take the operator $\Delta_q$ to the third
 equation, then
 \begin{equation}\label{3-6}\begin{array}{ll}\partial_t \Delta_q \psi + u\cdot \nabla \Delta_q \psi
 & ={\rm div}_R(-W(u) \cdot R \Delta_q \psi) + {\rm div}_R\left(\psi_\infty
 \nabla_R \left(\frac{\Delta_q \psi}{\psi_\infty}\right)\right)\\[3mm]&
 \ \ \ + {\rm div}_R~([\Delta_q, -W(u)]\cdot R\psi)
 + [u\cdot\nabla,
 \Delta_q]\psi\end{array}\end{equation}

Multiply (\ref{3-6}) by $r\left|\frac{\Delta_q
\psi}{\psi_\infty}\right|^{r-2} \frac{\Delta_q \psi}{\psi_\infty}$
and integrate over $B$,
$$\begin{array}{ll}
& \ \ \ \displaystyle\partial_t N_q^r + u\cdot \nabla N_q^r +
\frac{4(r-1)}{r}\int_{B}\psi_\infty \left|\nabla_R
\left(\frac{\Delta_q \psi}{\psi_\infty}\right)^{r/2}\right|^2
dR\\[4mm]
& \leq \displaystyle \int_B {\rm div}_R\left( [ \Delta_q, -W(u)]
\cdot R \psi\right) \cdot \left|\frac{\Delta_q
\psi}{\psi_\infty}\right|^{r-2}\frac{\Delta_q
\psi}{\psi_\infty}rdR\\[4mm] & \ \ \
+ \displaystyle\int_B |[u\cdot \nabla, \Delta_q]\psi|\cdot
\left|\frac{\Delta_q
\psi}{\psi_\infty}\right|^{r-1}r  dR\\
[4mm]& \triangleq J_1 + J_2.
\end{array}$$
By the Young's inequality and the H\"{o}lder inequality,
$$\begin{array}{ll}
|J_1| &  \displaystyle\leq C \|[\Delta_q, W(u)]\cdot
R\psi\|_{\mathcal{L}^r}^2 \cdot N_q^{r-2} + \frac{r-1}{r}\int_B
\psi_\infty \left|\nabla_R\left(\frac{\Delta_q
\psi}{\psi_\infty}\right)^{r/2}\right|^2dR,
\end{array}$$
$$J_2 \leq C\|[u\cdot \nabla, \Delta_q]\psi\|_{\mathcal{L}^r}\cdot N_q^{r-1}.$$
Hence
$$\partial_t N_q^r + u\cdot\nabla N_q^r
\leq C\|[\Delta_q, W(u)]\cdot R\psi\|_{\mathcal{L}^r}^2\cdot
N_q^{r-2} + C \|[u\cdot \nabla,
\Delta_q]\psi\|_{\mathcal{L}^r}\cdot N_q^{r-1},$$ which implies
that
\begin{equation}\label{3-7}\begin{array}{ll}2^{\alpha q r} \|N_q\|_{L^\infty}^r(t)
\leq & \displaystyle C 2^{\alpha q r} \int_0^t \|[\Delta_q,
W(u)]\cdot R\psi\|_{L^\infty_x(\mathcal{L}^r)}^2
\cdot \|N_q\|_{L^\infty}^{r-2} (s) ds\\
[3mm] & \displaystyle + C 2^{\alpha qr} \int_0^t \|[u\cdot\nabla,
\Delta_q]\psi\|_{L^\infty_x(\mathcal{L}^r)} \cdot
\|N_q\|_{L^\infty}^{r-1}(s)ds.\end{array}\end{equation} Note that
$$\begin{array}{ll} \left|[\Delta_q, W(u)]\cdot R\psi\right|&  =\displaystyle
\left|\int_{\mathbb{R}^2} h(y) [W(u)(x) - W(u)(x- 2^{-q}y) ]\cdot
(R\psi(x- 2^{-q}
y))dy\right|\\
[3mm] & \displaystyle \leq C 2^{-\alpha q}
\|W(u)\|_{\dot{C}^\alpha} \int_{\mathbb{R}^2}
|h(y)|\cdot|R\psi(x-2^{-q}y)|dy,
\end{array}$$ Then we get that
$$\|[\Delta_q, W(u)]\cdot R\psi\|_{L^\infty_x(\mathcal{L}^r)}
\leq C\|\nabla
u\|_{\dot{C}^\alpha}\|\psi\|_{L^\infty_x(\mathcal{L}^r)} \cdot
2^{-\alpha q}\leq C\|\nabla u\|_{\dot{C}^\alpha}\cdot 2^{-\alpha
q}.$$  And by the Bony's para-product formula,
$$\begin{array}{ll} |[u\cdot \nabla, \Delta_q]\psi| & \leq
\displaystyle \sum_{p\geq q-3}\sum_{|p-q^{\prime}|\leq 1}\left|
[\Delta_p u\cdot \nabla ,
\Delta_q ]\Delta_{q^{\prime}} \psi\right| \\[5mm] & \displaystyle+ \sum_{|q-
q^{\prime}|\leq 5}\left|[S_{q^{\prime}-1} u\cdot \nabla,
\Delta_q]\Delta_{q^{\prime}}\psi\right| + \displaystyle\sum_{|q-
q^{\prime}|\leq 5}\left| [\Delta_{q^{\prime}}u\cdot\nabla,
\Delta_q]S_{q^{\prime}-1}\psi\right|
\end{array}$$
with each term having the following estimate:

Since $$\begin{array}{ll}&[\Delta_p u\cdot\nabla,
\Delta_q]\Delta_{q^{\prime}} \psi(s, x, R)\\[3mm] =&
\displaystyle\int_{\mathbb{R}^2}h(y) [\Delta_p u(s,x) -
\Delta_pu(s, x-2^{-q}y)]\cdot \nabla \Delta_{q^{\prime}} \psi(s,
x-2^{-q}y, R) dy,\end{array}$$ and $$\nabla
\Delta_{q^{\prime}}\psi(x) = \displaystyle \int_{\mathbb{R}^2}
2^{2q^{\prime}}h(2^{q^{\prime}}y)\nabla\psi(x-y) dy =
-\int_{\mathbb{R}^2}2^{3q^{\prime}}\nabla
h(2^{q^{\prime}}y)\psi(x-y) dy$$ then
$$\begin{array}{ll} & \displaystyle\sum_{p\geq q-3}\sum_{|p-q^{\prime}|\leq 1}
 2^{\alpha q}\| [\Delta_p u\cdot
\nabla , \Delta_q ]\Delta_{q^{\prime}}
\psi\|_{L^\infty_x(\mathcal{L}^r)} (s) \\[5mm] \leq & C\displaystyle \sum_{p\geq q-3}
\sum_{|p-q^{\prime}|\leq 1}2^{\alpha q}2^{-q} \|\nabla \Delta_p
u\|_{L^\infty}(s)
\cdot \|\nabla\Delta_{q^{\prime}}\psi\|_{L^\infty_x(\mathcal{L}^r)}(s)\\[5mm]\leq
&C \displaystyle\sum_{p\geq q-3}\sum_{|p-q^{\prime}|\leq 1}
2^{\alpha q} 2^{-q}\|\nabla \Delta_p u\|_{L^\infty}(s)\cdot
2^{q^{\prime}}\|
\psi\|_{L^\infty_x(\mathcal{L}^r)}(s)\\
[5mm]\leq & C\displaystyle \sum_{p\geq q-3}2^{\alpha p} \|\nabla
\Delta_p u\|_{L^\infty}(s)\cdot
2^{\alpha(q-p)}\|\psi\|_{L_x^\infty(\mathcal{L}^r)}(s)
\\[4mm] \leq & CA(s) \cdot
\|\psi\|_{L^\infty_x(\mathcal{L}^r)}(s)\leq CA(s).
\end{array}$$
Similarly,
$$\begin{array}{ll}& \displaystyle\sum_{|q-q^{\prime}|\leq
5}2^{\alpha q} \|[S_{q^{\prime}-1} u\cdot\nabla,
\Delta_q]\Delta_{q^{\prime}} \psi\|_{L^\infty_x(\mathcal{L}^r)}(s) \\[5mm]\leq
& \displaystyle C\sum_{|q-q^{\prime}|\leq 5}2^{\alpha q-q}
\|\nabla S_{q^{\prime}-1} u\|_{L^\infty}(s)\|\nabla
\Delta_{q^{\prime}}
\psi\|_{L^\infty_x(\mathcal{L}^r)}(s)\\[5mm] \leq & \displaystyle
C\sum_{|q-q^{\prime}|\leq 5}\|\nabla u\|_{L^\infty}(s) \cdot
2^{\alpha q^{\prime}}
\|\Delta_{q^{\prime}}\psi\|_{L^\infty_x(\mathcal{L}^r)}(s)\\[5mm]
\leq & C\|\nabla u\|_{L^\infty}(s)\cdot D(s).
\end{array}$$
$$\begin{array}{ll}&
\displaystyle\sum_{|q^{\prime}-q|\leq 5}2^{\alpha q}
\|[\Delta_{q^{\prime}}u\cdot\nabla, \Delta_q]S_{q^{\prime}
-1}\psi\|_{L^\infty_x(\mathcal{L}^r)}(s)\\[5mm]
\leq & C \displaystyle \sum_{|q^{\prime}-q|\leq 5}2^{\alpha q -
q}\|\nabla
\Delta_{q^{\prime}}u\|_{L^\infty}(s)\cdot\|S_{q^{\prime}-1}\nabla
\psi\|_{L^\infty_x(\mathcal{L}^r)}(s)\\[5mm]
\leq & C A(s)\| \psi\|_{L^\infty_x(\mathcal{L}^r)}(s)\leq CA(s).
\end{array}$$
Therefore, taking the supreme of (\ref{3-7}) with respect to $q$,
by the Young's inequality and (\ref{3-5}),
$$\begin{array}{ll}
D(t)^r & \leq \displaystyle\int_0^t C\left[A(s)^r +
D(s)^r\right] ds + C\int_0^t\|\nabla u\|_{L^\infty}(s) D(s)^r ds\\
& \leq \displaystyle \int_0^t C\left[1+D(s)^r\right]+ C\|\nabla
u\|_{L^\infty}D(s)^r ds
\end{array}$$
Then the Gronwall's inequality implies that
$$D(t)\leq C(D(0)+1) e^{C\int_0^t (\|\nabla u\|_{L^\infty} + 1)ds}.$$

Hence according to Lemma \ref{lemma 2.2}, $$\begin{array}{ll} e +
D(t) & \leq C(D(0) + 1)C_*\exp\{C
\int_{t_*}^t( \|\nabla u\|_{L^\infty}+1)ds\}\\[3mm]
& \leq CC_* (D(0)+1) \exp\{C(1+ \int_{t_*}^t \|u\|_{L^2}ds+
\epsilon \ln(e+ (t-t_*)A(t)) )\}\\[3mm]
& \leq CC_*(D(0)+1)\exp\{C[1+ \epsilon \ln(e+D(t))]\}\\[3mm]
&\leq CC_* (D(0) +1)(e+ D(t))^{C\epsilon},\end{array}$$ where
$C_*$ is some positive constant depending on the solution $u$ on
$[0, t_*]$. Choosing $\epsilon = \frac{1}{2C}$,
$$D(T) \leq [CC_*(D(0)+1)]^2.$$
Then by (\ref{3-5}), $A(T)$ is bounded, which implies that
$\|\nabla u\|_{L^\infty}$ is bounded on $[0,T]$ since
$$\|\nabla u\|_{L^\infty}\leq C\left(\|u\|_{L^2}+ \|u\|_{\dot{C}^{1+\alpha}}\right).$$

\section{Proof of Theorem 1.2}

The proof of Theorem \ref{thm 1.2} is similar to that of Theorem
\ref{thm 1.1}, we just give the sketch. As above, to get the
global existence we only need to control $\|\nabla
u\|_{L^\infty},$ see \cite{CMa}.

\vspace{2mm}\textbf{Step I~~Uniform estimates for $\psi$ and
$\tau$}

Integrate the third equation of (\ref{1.2}) on $M$, then
$$\partial_t \int_M\psi(t,x,m)dm + u\cdot \nabla\int_M \psi(t,x,m) dm
=0.$$ By the maximum principle for evolutionary equation, $\psi$
is always nonnegative. Hence
\begin{equation}\label{4.1}\|\psi\|_{L_x^1\cap L_x^\infty(L^1(M))}(t) =
\|\psi_0\|_{L_x^1\cap L_x^\infty(L^1(M))}\end{equation}
\begin{equation}\label{4.2}\|\tau\|_{L^\infty(0,T; L^2)}\leq
C\left(\|\psi_0\|_{L_x^4(L^1(M))}^2 +
\|\psi_0\|_{L_x^2(L^1(M))}\right),\end{equation}
\begin{equation}\label{4.3} \|\tau\|_{L^\infty(0,T; L^\infty)}\leq
C(\|\psi_0\|_{L_x^\infty(L^1(M))}^2 + 1)\end{equation} Since
$s>\frac{d}{2} +1$ and $M$ is a smooth compact manifold without
boundary, \begin{equation}\label{4.4}
\|\psi\|_{L_{t,x}^\infty(H^{-s}(M))}\leq
C\|\psi_0\|_{L_x^\infty(L^1(M))}.\end{equation}

\vspace{2mm}\textbf{Step II~~A priori estimates for $u$}

Applying Lemma \ref{lemma 3.2} and the estimates
(\ref{4.2})(\ref{4.3}), we get that
$$u\in L^\infty(0,T; L^2)\cap L^2(0,T; \dot{H}^1)\cap \tilde{L}^1(0,T; \dot{C}^1)$$
and $\forall\ \epsilon >0$, there exists $t_0(\epsilon)\in (0,T),$
such that $\|u\|_{\tilde{L}^1(t_0, T; \dot{C}^1)}\leq \epsilon.$

\vspace{2mm}\textbf{Step III~~H\"{o}lder estimates for $u$}

Denote $$H= (-\Delta_g + I)^{-\frac{s}{2}}, \ \ \ N_q^2(t,x) =
\int_M |H\Delta_q \psi(t,x,m)|^2 dm,$$
$$A(t) = \sup_{0\leq s <t}\|u(s,\cdot)\|_{\dot{C}^{1+\alpha}},\ \
\ B(t)= \sup_{0\leq s<t }\|\tau(s, \cdot)\|_{\dot{C}^\alpha},$$
$$D(t) = \sup_{0\leq s<t }\sup_{q\in\mathbb{Z}}2^{\alpha q}\|N_q(s,\cdot)\|_{L^\infty}.$$
As in section 3, we have the estimates
$$A(t)\leq C \left(\|u_0\|_{\dot{C}^{1+\alpha}} + B(t)\right)\leq C(1+B(t)),$$
and by (\ref{4.4}), $\forall q\in \mathbb{Z}$,
$$\|N_q(t,\cdot)\|_{L^\infty}\leq C\|H\psi(t, \cdot,
\cdot)\|_{L_x^\infty(L^2(M))}\leq
C\|\psi_0\|_{L_x^\infty(L^1(M))},$$
$$\begin{array}{l}\ \ \ \ 2^{\alpha q}\|\Delta_q \tau_{ij}(s,x)\|_{L^\infty}
\\[3mm] \leq 2^{\alpha q}\left\|\int_M\int_M
H_{m_1}^{-1}H_{m_2}^{-1} \gamma_{ij}^{(2)}\Delta_q
(H_{m_1}\psi(s,x,m_1) H_{m_2}\psi(s,x,m_2))dm_1
dm_2\right\|_{L^\infty}\\[3mm] \ \  +2^{\alpha q}\left\|\int_M H^{-1}\gamma_{ij}^{(1)}(m)
H\Delta_q\psi(s,x,m)dm\right\|_{L^\infty} \\[3mm] \leq \displaystyle
2^{\alpha q}\sum_{|p-q|\leq 5}\left\|\int_M\int_M
H_{m_1}^{-1}H_{m_2}^{-1} \gamma_{ij}^{(2)}(m_1, m_2)
S_{p-1}H\psi(m_1) \Delta_p
H\psi(m_2) dm_1 dm_2 \right\|_{L^\infty}\\[4mm]
\ \  +\displaystyle 2^{\alpha q}\sum_{|p-q|\leq
5}\left\|\int_M\int_M H_{m_1}^{-1}H_{m_2}^{-1}
\gamma_{ij}^{(2)}(m_1,m_2) \Delta_p H\psi(m_1)S_{p-1}H\psi(m_2)
dm_1dm_2\right\|_{L^\infty}\\[3mm] \ \ +\displaystyle 2^{\alpha q}\sum_{p\geq
q-3}\sum_{|p-r|\leq 1} \left\| \int_M\int_M
H_{m_1}^{-1}H_{m_2}^{-1}\gamma_{ij}^{(2)}
\Delta_pH\psi(m_1)\Delta_rH\psi(m_2)
dm_1dm_2\right\|_{L^\infty}\\[4mm]\ \ + C 2^{\alpha
q}\|N_q(s,\cdot)\|_{L^\infty}\\[2mm]
\leq CD(s) \|H\psi\|_{L_x^\infty(L^2(M))}(s)+ CD(s)\leq CD(s)
\end{array}$$
which implies that $$B(t)\leq CD(t).$$ Therefore,
\begin{equation}\label{4.5}
A(t)\leq C(1+D(t)).
\end{equation}

\vspace{2mm}\textbf{Step IV~~H\"{o}lder estimates for $\psi$}

Take the operator $H$ and $\Delta_q$ to the third equation,
multiply by $\Delta_q H\psi$ and integrate over $M$, then
$$\begin{array}{l}\ \ \ \displaystyle \frac{1}{2}\partial_t\int_{M}|\Delta_q
H\psi|^2 dm + \frac12 u\cdot \nabla \int_{M}|\Delta_q H\psi|^2dm +
\int_{M}|\nabla_g
\Delta_q H\psi|^2dm \\[3mm]=\displaystyle \int_M [u\cdot \nabla, \Delta_q]H\psi\cdot \Delta_q H\psi dm
-\int_M \Delta_q H {\rm div}_g~(G(u,\psi)\psi)\cdot \Delta_q H\psi
dm\\[3mm]
= \displaystyle \int_M [u\cdot \nabla, \Delta_q]H\psi\cdot
\Delta_q H\psi dm -\partial_j u_i\int_M H{\rm
div}_g~(c_\alpha^{ij}\Delta_q
\psi)\cdot H\Delta_q \psi dm \\[4mm] \ \ \ \ \displaystyle +\int_M [\partial_j u_i, \Delta_q
]H(c_\alpha^{ij}\psi)\cdot \nabla_g\Delta_q H\psi dm +\int_M
\Delta_q
(\nabla_g \mathcal{U} H\psi) \cdot \nabla_g \Delta_q H\psi dm\\[4mm]\ \ \ \
+\displaystyle \int_M \Delta_q [H\nabla_g \mathcal{U},
H^{-1}]H\psi\cdot \nabla_g \Delta_q H\psi dm
\end{array}$$ By the Young's inequality, we have
\begin{equation}\label{4-add}\begin{array}{l}\ \ \ \ \ \ \frac12\|\Delta_q
H\psi\|_{L_x^\infty(L^2(M))}^2(t)
\\[3mm]\leq \displaystyle 2\int_0^t \left\| \int_M [u\cdot\nabla, \Delta_q]
H\psi\cdot \Delta_q H\psi dm \right\|_{L^\infty}ds \\[3mm]\ +\displaystyle \int_0^t
\left\|\partial_j u_i\int_M H{\rm div}_g~(c_\alpha^{ij} \Delta_q
\psi)\Delta_q H\psi dm \right\|_{L^\infty}ds\\[3mm]\ + \displaystyle\frac34\int_0^t
\left\|[\partial_j u_i, \Delta_q]
H(c_\alpha^{ij}\psi)\right\|_{L^\infty_x(L^2(M))}^2ds\\[3mm]\displaystyle\ +\frac34
\int_0^t\left\| \Delta_q
(\nabla_g\mathcal{U}H\psi)\right\|_{L^\infty_x(L^2(M))}^2ds
+\frac34
\int_0^t\left\|\Delta_q ([H\nabla_g \mathcal{U}, H^{-1}]H\psi)\right\|_{L^\infty_x(L^2(M))}^2ds\\[4mm]
\triangleq\displaystyle\int_0^t (J_1 + J_2 + J_3 + J_4 +J_5) ds.
\end{array}\end{equation}
As in section 3, applying (\ref{4.5}), for every $q\in
\mathbb{Z}$,
$$\begin{array}{ll}2^{2\alpha q}J_1(s) &\leq CA(s)\|H\psi\|_{L_x^\infty(L^2(M))}\cdot D(s)
+ C\|\nabla u\|_{L^\infty}\cdot D(s)^2 \\[2mm]&\leq C(D(s)^2+1)(\|\nabla
u\|_{L^\infty}(s) + 1).\end{array}$$ $J_2$, $J_4$ and $J_5$ are
estimated as in \cite{CFTZ},
$$\begin{array}{ll}&\displaystyle \left|\int_M H {\rm div}_g (c_{\alpha}^{ij}\Delta_q \psi )
\cdot \Delta_q H\psi dm\right|\\[3mm]
\leq & \displaystyle\left|\int_M {\rm
div}_g~(c_\alpha^{ij}H\Delta_q \psi)\cdot \Delta_q H\psi
dm\right|+ \left|\int_M \left([H{\rm div}_g c_\alpha^{ij},
H^{-1}]H\Delta_q \psi\right)\cdot \Delta_q
H\psi dm\right|\\[4mm]\leq & \displaystyle \left|
\int_M \frac{1}{2}({\rm div}_g c_\alpha^{ij})|\Delta_q H\psi|^2
dm\right|+ \left|\int_M \left([H{\rm div}_g c_\alpha^{ij},
H^{-1}]H\Delta_q \psi\right)\cdot \Delta_q H\psi
dm\right|\\[4mm]
\leq & C\|\Delta_q H\psi\|_{L^2(M)}^2,\end{array}$$which deduces
that
$$2^{2\alpha q}J_2(s)\leq C 2^{2\alpha q}\|\nabla u\|_{L^\infty}(s)\cdot N_q^2(s)
\leq C\|\nabla u\|_{L^\infty}(s)\cdot D(s)^2.$$ Using Bony's
decomposition,
$$\begin{array}{l}\ \ \ \ 2^{\alpha q}\|\Delta_q (\nabla_g \mathcal{U} H\psi)\|_{L^\infty_x(L^2(M))}
\\[2mm]\leq 2^{\alpha q}\displaystyle \sum_{|p-q|\leq
5}\left\|S_{p-1}\nabla_g \mathcal{U} \cdot \Delta_p
H\psi\right\|_{L_x^\infty(L^2(M))} \\[4mm]\ \ \ \ \displaystyle + \sum_{|p-q|\leq
5}\left\|\Delta_p \nabla_g \mathcal{U}\cdot
S_{p-1}H\psi\right\|_{L_x^\infty(L^2(M))}\\[4mm]
\displaystyle\ \ \ \ + \sum_{p\geq q-3}\sum_{|p-r|\leq
1}\left\|\Delta_p \nabla_g \mathcal{U} \cdot \Delta_r
H\psi\right\|_{L_x^\infty(L^2(M))}\\[3mm]
\displaystyle\leq C\|\nabla_g
\mathcal{U}\|_{L_x^\infty(L^2(M))}\cdot \sup_p 2^{\alpha
p}\|\Delta_p H\psi\|_{L_x^\infty(L^2(M))}\\[3mm]\ \ \ \ +\displaystyle C\sup_p 2^{\alpha p}
\|\Delta_p \nabla_g \mathcal{U}\|_{L_x^\infty(L^2(M))}\cdot
\|H\psi\|_{L^\infty_x(L^2(M))}\end{array}$$ Combining the
relationship between $\mathcal{U}$ and $\psi$,
 $$2^{2\alpha q} J_4(s) \leq C2^{2\alpha q}\|H\psi\|_{L_x^\infty(L^2(M))}^2(s) \cdot N_q^2(s)\leq
 CD(s)^2.$$
 Similarly,
$$2^{2\alpha q} J_5(s)\leq C\|H\psi\|_{L_x^\infty(L^2(M))}^2(s) \cdot \sup_p 2^{2\alpha p}N_p^2(s)\leq
 CD(s)^2.$$
 Since $$[\partial_j u_i, \Delta_q]H(c_\alpha^{ij}\psi)= \int_{\mathbb{R}^2}
 h(y)[\partial_j u_i(x) - \partial_j u_i(x-2^{-q}y)]H(c_\alpha^{ij}\psi)(x-2^{-q}y, m)dy,$$
$$\begin{array}{ll}2^{2\alpha q}J_3(s) &\leq
C2^{2\alpha q}\cdot 2^{-2\alpha q}\|\nabla
u\|_{\dot{C}^\alpha}^2(s)\|H(c_\alpha^{ij}\psi)\|_{L^\infty_x(L^2(M)))}^2(s)\\[3mm]&
\leq CA(s)^2 \left[\|c_\alpha^{ij} H\psi\|_{L^\infty_x(L^2(M))}^2
+ \|H[c_\alpha^{ij},
H^{-1}]H\psi\|_{L^\infty_x(L^2(M))}^2\right]\\[3mm]&\leq CA(s)^2
\|H\psi\|_{L^\infty_x(L^2(M))}^2\\[2mm] &\leq CA(s)^2.
\end{array}
 $$

Therefore, taking the supreme of $(\ref{4-add})\times 2^{2\alpha
q}$ with respect to $t$ and $q$,
$$D(t)^2 \leq C\int_0^t (\|\nabla u\|_{L^\infty}+1)(1+ D(s)^2) ds.$$
The remaining proof is just the same as that in section 3. We omit
the details.

\section*{Acknowledgement}
The work was in part supported by NSFC (grants No. 10801029 and
10911120384), FANEDD, Shanghai Rising Star Program (10QA1400300)
and SGST 09DZ2272900. Part of the work was done when Zhen Lei was
visiting the Institute of Mathematical Sciences of CUHK. Zhen Lei
would like to thank the hospitality of Professor Zhouping Xin and
the Institute.


\vspace{2mm}
\begin{thebibliography}{100}
\bibitem{B}{\sc J. M. Bony}, {\em Calcul symbolique et propagation des singularit\'{e}s pour les
\'{e}quations aux d\'{e}riv\'{e}es partielles non lin\'{e}aires.},
Ann. Sci. \'{E}cole Norm. Sup. (4), 14(1981), no. 2, pp.~209-246.

\bibitem{BC1994} {\sc H. Bahouri and J.-Y. Chemin},
{\em ${\rm \acute{e}}$quations de transport relatives ${\rm
\acute{a}}$ des champs de vecteurs non-lipschitziens et m${\rm
\acute{e}}$canique des fluides. (French) [Transport equations for
non-Lipschitz vector fields and fluid mechanics]},  Arch. Rational
Mech. Anal., 127(1994), pp.~159--181.

\bibitem{BCA} {\sc R. B. Bird, C. F. Curtis, R. C. Armstrong and O. Hassager},
{\em Dynamics of Polymeric Liquids}, Vol. 2: Kinetic Theory,
Weiley Interscience, New York, 1987.

\bibitem{BSS}{\sc J. W. Barrett, C. Schwab and E. S\"{u}li}, {\em Existence of
global weak solutions for some polymeric flow models}, Math.
Models Methods Appl. Sci., 15(2005), no. 6, pp.~939--983.

\bibitem{Ch}{\sc J. Y. Chemin}, {\em Perfect Incompressible Fluids},
Oxford Lectures Series in Mathematics and its Applications, Vol.
14, Clarendon Press Oxford University Press, New York, 1998.

\bibitem{CZ}{\sc Y. Chen and P. Zhang},  {\em The global existence of small
solutions to the incompressible viscoelastic fluid system in 2 and
3 space dimensions},  Comm. Partial Differential Equations, 31
(2006), no. 10-12, pp.~1793--1810.

\bibitem{CM}{\sc J. Y. Chemin and N. Masmoudi}, {\em About lifespan of regular
solutions of equations related to viscoelastic fluids}, SIAM J.
Math. Anal.,  33 (2001), no. 1, pp.~84--112.

\bibitem{C}{\sc P. Constantin}, {\em Nonlinear Fokker-Planck Navier-Stokes
systems},  Commun. Math. Sci.,  3 (2005), no. 4, pp.~531--544.

\bibitem{CFTZ}{\sc P. Constantin, C. Fefferman, E. Titi and A. Zarnescu},
{\em Regularity for coupled two-dimensional nonlinear
Fokker-Planck and Navier-Stokes system}, Commun. Math. Phys.,
270(2007), pp. ~789--811.

\bibitem{CMa}{\sc P. Constantin and N. Masmoudi}, {\em Global well-posedness
for a Smoluchowski equation coupled with Navier-Stokes Equations
in 2D},  Commun. Math. Phys.,  278 (2008), pp.~179--191.

\bibitem{DE}{\sc M. Doi and S. F. Edwards}, {\em The Theory of Polymeric
Dynamics}, Oxford University Press, Oxford, UK, 1986.

\bibitem{ELZ}{\sc W. N. E, T. J. Li and P. W. Zhang}, {\em Well-posedness for
the dumbbell model of polymeric fluids}, Commun. Math. Phys.,
248(2004), no. 2, pp.~409--427.

\bibitem{JLL}{\sc B. Jourdain, T. Leli\`{e}vre and C. Le Bris}, {\em Existence
of solutions for a micro-macro model of polymeric fluid: the FENE
model},  J. Funt. Anal.,  209 (2004), no. 1, pp.~162--193.

\bibitem{Lei1}{\sc Z. Lei},  {\em Global existence of classical solutions
for some Oldroyd-B model via the incompressible limit}, Chinese
Ann. Math. Ser. B,  27(2006), no. 5, pp.~565--580.

\bibitem{Lei2}{\sc Z. Lei}, {\em On 2D viscoelasticity with small strain},
 Arch. Rational Mech. Anal, to appear. {\it arXiv: 0904.1345v1}

\bibitem{LLZh}{\sc Z. Lei, C. Liu and Y. Zhou}, {\em Global solutions for
incompressible viscoelastic fluids}, Arch. Ration. Mech. Anal, 188
(2008), no. 3, pp.~371--398.

\bibitem{LMZ}{\sc Z. Lei, N. Masmoudi and Y. Zhou}, {\em Remarks on the blowup
criteria for Oldroyd models}, J. Differential Equations, 248, no.
2, pp.~328--341.

\bibitem{LZ}{\sc Z. Lei, and Y. Zhou}, {\em Global existence of classical solutions for
the two-dimensional Oldroyd model via the incompressible limit},
SIAM J. Math. Anal, 37(2005), no. 3, pp.~797--814.

\bibitem{LLZ1}{\sc F. H. Lin, C. Liu and P. Zhang}, {\em On hydrodynamics of viscoelastic
fluids}, Comm. Pure Appl. Math., 58 (2005), no. 11,
pp.~1437--1471.

\bibitem{LLZ}{\sc F. H. Lin, C. Liu and P. Zhang}, {\em On a micro-macro model
for polymeric fluids near equilibrium}, Comm. Pure Appl. Math., 60
(2007), no. 6, pp.~838--866.

\bibitem{LM}{\sc P. L. Lions and N. Masmoudi}, {\em Global existence of weak
solutions to micro-macro models}, C. R. Math. Acad. Sci. Paris,
345 (2007), no. 1, pp.~15--20.

\bibitem{LZZ}{\sc F. H. Lin, P. Zhang and Z. F. Zhang}, {\em On the global
existence of smooth solutions to the 2-d FENE dumbbell model},
Comm. Math. Phys., 277 (2008), no. 2, pp.~531--553.

\bibitem{Ma}{\sc N. Masmoudi}, {\em Well-posedness
 for the FENE dumbbell model of polymeric flows}, Comm. Pure Appl. Math., 61 (2008),
  no. 12, pp.~1685--1714.

\bibitem{Masmoudi2}{\sc N. Masmoudi}, {\em Global existence of weak solutions
to the FENE dumbbell model of polymeric flows}, avaliable at:
arXiv:1004.4015

\bibitem{ZZ}{\sc H. Zhang and P. Zhang}, {\em Local existence for the FENE-dumbbell model of polymeric fluids},
Arch. Ration. Mech. Anal., 181 (2006), no. 2, pp.~373-400.


\end{thebibliography}
\end{document}